\documentclass[12pt, twoside, leqno]{article}

\usepackage{amsmath,amsthm}
\usepackage{amssymb}

\usepackage{enumitem}

\usepackage{graphicx}

\usepackage[T1]{fontenc}

\usepackage{listings}
\usepackage{amsthm}
\usepackage{amsmath}
\usepackage{amssymb}
\usepackage{amsfonts}
\usepackage{fancyhdr}
\usepackage{enumerate}
\usepackage{setspace}
\usepackage{cite}
\usepackage{todonotes}
\usepackage[english]{babel}
\usepackage [autostyle, english = american]{csquotes}
\MakeOuterQuote{"}
\allowdisplaybreaks

\newtheorem{theorem}{Theorem}[section]
\newtheorem{corollary}[theorem]{Corollary}
\newtheorem{lemma}[theorem]{Lemma}
\newtheorem{proposition}[theorem]{Proposition}

\newtheorem{question}[theorem]{Question}
\theoremstyle{definition}
\newtheorem{definition}[theorem]{Definition}
\newtheorem{remark}[theorem]{Remark}
\newtheorem{example}[theorem]{Example}

\numberwithin{equation}{section}

\begin{document}


\baselineskip=17pt


\title{Amenability constants of central Fourier algebras of finite groups}

\author{John Sawatzky\\
University of Waterloo\\
Department of Pure Mathematics\\
Waterloo, Ontario, N2L 3G1, Canada\\
E-mail: jmsawatz@uwaterloo.ca
}

\date{}

\maketitle


\renewcommand{\thefootnote}{}



\renewcommand{\thefootnote}{\arabic{footnote}}
\setcounter{footnote}{0}


\begin{abstract}
    We consider amenability constants of the central Fourier algebra $ZA(G)$ of a finite group $G$. This is a dual object to $ZL^1(G)$ in the sense of hypergroup algebras, and as such shares similar amenability theory. We will provide several classes of groups where $AM(ZA(G)) = AM(ZL^1(G))$, and discuss $AM(ZA(G))$ when $G$ has two conjugacy class sizes. We also produce a new counterexample that shows that unlike $AM(ZL^1(G))$, $AM(ZA(G))$ does not respect quotient groups, however the class of groups that does has $\frac{7}{4}$ as the sharp amenability constant bound.
\end{abstract}
\newpage
\section{Background and Notation}
\subsection{Banach Algebra Amenability}
The notion of amenability of a Banach algebra was first introduced by B.E. Johnson in \cite{Johnson1}, and has been a fruitful area of research ever since. A brief introduction is offered here, but we recommend \cite{Runde2} for a more detailed explanation. Let $\mathcal{A}$ be a Banach algebra. We denote the projective tensor product of $\mathcal{A}$ with itself by $\mathcal{A} \hat{\otimes} \mathcal{A}$. There are natural module actions of $\mathcal{A}$ on $\mathcal{A} \hat{\otimes} \mathcal{A}$ that are given by $a \cdot (b \otimes c) = (ab) \otimes c$ and $(b \otimes c) \cdot a = b \otimes (ca)$, where $a, b, c \in \mathcal{A}$. The multiplication map $m: \mathcal{A} \hat{\otimes} \mathcal{A} \rightarrow \mathcal{A}$ acts on elementary tensors by $m(a \otimes b) = ab$. 
\begin{definition}
 A \textit{bounded approximate diagonal} (b.a.d.) for $\mathcal{A}$ is a bounded net $(d_\alpha)_\alpha$ in $\mathcal{A} \hat{\otimes} \mathcal{A}$ such that $a \cdot d_\alpha - d_\alpha \cdot a \rightarrow 0$ and $a m(d_\alpha) \rightarrow a$, where $a \in \mathcal{A}$. 
\\
Similarly, a \textit{virtual diagonal} for $\mathcal{A}$ is defined as an element $D \in (\mathcal{A} \hat{\otimes} \mathcal{A})^{**}$ such that $a \cdot D = D \cdot a$ and $a \cdot m^{**}D = a$, where $a \in \mathcal{A}$.
\end{definition}
\begin{remark}
A copy of $\mathcal{A} \hat{\otimes} \mathcal{A}$ can be identified as living inside of $(\mathcal{A} \hat{\otimes} \mathcal{A})^{**}$. If there is a virtual diagonal which is an element of $\mathcal{A} \hat{\otimes} \mathcal{A} \subseteq (\mathcal{A} \hat{\otimes} \mathcal{A})^{**}$ then we call it a \textit{diagonal}. In the case that $\mathcal{A}$ is a finite-dimensional commutative amenable Banach algebra, then $\mathcal{A}$ possesses a unique diagonal \cite{Ghandehari}.
\end{remark}
For a given Banach algebra $\mathcal{A}$, the existence of a bounded approximate diagonal is equivalent to the existence of a virtual diagonal. If this condition is satisfied then we say that $\mathcal{A}$ is an \textit{amenable Banach algebra}. The term "amenable" here does have a relationship with amenability in groups, as demonstrated by the pleasing result of Johnson \cite{Johnson1} that the group algebra $L^1(G)$ of a locally compact group $G$ is amenable if and only if $G$ is amenable, in which case $L^1(G)$ has an approximate diagonal with norm bounded by $1$. This example begs the question of whether every amenable Banach algebra automatically has a b.a.d. bounded by $1$, but as implied by the next definition this need not be the case.
\begin{definition}
For a Banach algebra $\mathcal{A}$ we denote the \textit{amenability constant} of $\mathcal{A}$ by
\begin{center}
$AM(\mathcal{A}) = \inf \{ \sup\limits_{\alpha} \|\omega_\alpha\|: (\omega_\alpha) $ is a b.a.d. for $\mathcal{A}\}    $.
\end{center}
If $\mathcal{A}$ is not amenable then we let $AM(\mathcal{A}) = \infty$.
\end{definition}
It is clear that amenability constants, if they exist, must be greater or equal to $1$. We will largely be interested in the behavior of amenability constants that are close but not equal to $1$.
\begin{definition}
A class of Banach algebras $\mathcal{C}$ has an \textit{amenability gap} if there exists $\lambda > 1$ such that for $A \in \mathcal{C}$, $AM(A) = 1$ if and only if $AM(A) < \lambda$. Any $\lambda$ with this property is called an \textit{amenability bound}, and if $\lambda$ is the supremum of all possible amenability bounds then it is called the \textit{sharp bound}.
\end{definition}
\subsection{Fourier Algebra}
The Fourier algebra $A(G)$ of a locally compact group $G$ is a Banach algebra that was first introduced by Eymard in 1964 in \cite{Eymard}. A modern introduction to this algebra can be found in \cite{Kaniuth}. The amenability constant theory of $A(G)$ was originally studied in \cite{Johnson2}, where it was shown that if $G$ is finite then
\[
AM(A(G)) = \frac{1}{|G|}\sum\limits_{\chi \in \textrm{Irr}(G)} d_\chi^3.
\]
Representation theory proves to be an invaluable tool for studying amenability constants; a good reference for character theory on finite groups is \cite{Isaacs}. It follows from the basic fact that $\sum\limits_{\chi \in \textrm{Irr}(G)} d_\chi^2 = |G|$ that $AM(A(G)) = 1$ if and only if $G$ is abelian. While this is less obvious if $G$ is infinite, it was shown by Runde in \cite{Runde1} that this property holds for all locally compact groups. In a recent preprint by Choi \cite{Choi2}, it was proven that $\frac{3}{2}$ is the sharp amenability bound for $A(G)$ for locally compact $G$. A key feature of the amenability theory of $A(G)$ is that if $H \leqslant G$ is a closed subgroup then $AM(A(H)) \leq AM(A(G))$. This comes from the important fact that the restriction map on the Fourier algebra is surjective. 
\subsection{$ZL^1(G)$ and $ZA(G)$}
While $L^1(G)$ does not have an amenability constant gap, it turns out that its centre, $ZL^1(G)$, does. From here on we will assume that all groups $G$ are finite. We can write $ZL^1(G) = \overline{\text{span}}^{L^1(G)} \textrm{Irr}(G)$, so the algebra depends only on the character table of $G$. By \cite{Azimifard} $AM(ZL^1(G))$ can be calculated using the following nice formula:
\[
AM(ZL^1(G)) = \frac{1}{|G|^2} \sum_{C,C' \in \textrm{Conj}(G)} |C| |C'| \left|\sum_{\chi \in \textrm{Irr}(G)}d_\chi^2 \chi_\pi(C)\overline{\chi_\pi(C')}\right|.
\]
If we take the $A(G)$ norm of the irreducible characters instead of the $L^1$ norm, we get the central Fourier algebra $ZA(G) = \overline{\text{span}}^{A(G)} \textrm{Irr}(G)$. As sets of functions both $ZL^1(G)$ and $ZA(G)$ form the class functions on $G$, albeit with different norms. Amenability properties of $ZA(G)$ were studied in \cite{Alag2}, where it was shown that the amenability constant can be calculated using a formula that is dual to the $ZL^1(G)$ case:
\[
AM(ZA(G)) = \frac{1}{|G|^2}\sum_{\chi, \chi' \in \textrm{Irr}(G)}d_\chi d_{\chi'} \left |\sum_{C \in \textrm{Conj}(G)} |C|^2 \chi(C)\overline{\chi'(C)}\right|.
\]
For convenience we will set $AMZA(G) = AM(ZA(G))$ and $AMZL(G) = AM(ZL^1(G))$. Note that if you choose to only sum up the elements where $\chi = \chi'$ in the above formula then you get the same value as you get from only summing up the elements where $C = C'$ in the formula for $AMZL(G)$. We will call this quantity the \textit{auxiliary minorant} of $AMZA(G)$ (or equivalently, of $AMZL(G)$). Matching the notation from \cite{Choi1}, we denote it as follows:
\[
\textrm{ass}(G) = \frac{1}{|G|^2}\sum_{\chi \in \textrm{Irr}(G)}d_\chi^2 \sum_{C \in \textrm{Conj}(G)} |C|^2 |\chi(C)|^2.
\]
It will often make sense to split up a calculation between $\textrm{ass}(G)$ and $AMZA(G) - \textrm{ass}(G)$, which we denote as $AMZA_{\textrm{off}}(G)$, and can be written as
\[
AMZA_{\textrm{off}}(G) = \frac{1}{|G|^2}\sum_{\chi \neq \chi' \in \textrm{Irr}(G)}d_\chi d_{\chi'} \left|\sum_{C \in \textrm{Conj}(G)} |C|^2 \chi(C)\overline{\chi'(C)}\right|.
\]
The classes of Banach algebras described by $ZL^1(G)$ and $ZA(G)$ for finite $G$ are both known to possess amenability gaps. The first of these was shown in 2008 by Aimifard, Samei, and Spronk in \cite{Alag3}, which used a result of Rider \cite{Rider} to get that $1 + \frac{1}{300}$ is an amenability bound for $ZL^1(G)$. In 2016 Alaghmandan and Spronk \cite{Alag2} used methods involving completely bounded multipliers to get an amenability bound of $\frac{2}{\sqrt{3}}$ for $ZA(G)$. Shortly afterwords Choi \cite{Choi1} resolved the question in the case of $ZL^1(G)$ by proving that $\frac{7}{4}$ is the sharp amenability bound. The sharp amenability bound has not yet been found for $ZA(G)$, although computational results suggest that $\frac{7}{4}$ will be the sharp bound just like with $ZL^1(G)$. Similarly to $A(G)$, both $ZL^1(G)$ and $ZA(G)$ possess the property that the amenability constant is equal to $1$ if and only if $G$ is abelian.
\\\\
The duality between $ZA(G)$ and $ZL^1(G)$ is no coincidence, the amenability properties of these algebras are deeply intrinsically linked. As shown in \cite{Alag4}, they can be understood as dual hypergroup algebras, specifically $ZL^1(G) \cong \ell^1(\textrm{Conj}, \lambda_{\textrm{Conj}})$ and $ZA(G) \cong \ell^1(\hat{G}, \lambda_{\hat{G}})$, where $\lambda_{\textrm{Conj}}(C) = |C|$ and $\lambda_{\hat{G}}(\pi) = d_\pi^2$. Their amenability constant formulas are just special cases of the following:
\begin{theorem}[{\cite[Theorem 3.7]{Alag4}}]
Let $H$ be a finite commutative hypergroup with Haar measure $\lambda$. For $\chi \in \hat{H}$ let $k_\chi$ denote the hyperdimension of $\chi$. Then we have that
\[
AM(\ell^1(H,\lambda)) = \frac{1}{\lambda(H)^2}\sum\limits_{x,y \in H} \left|\sum_{\chi \in \hat{H}} k_\chi^2 \chi(x) \overline{\chi(y)}\right| \lambda(x) \lambda(y).
\]
\end{theorem}
Given that $ZA(G)$ is an algebra that is deeply tied to $ZL^1(G)$ and $A(G)$, it is natural to hope that their nice properties would also hold for $ZA(G)$. This unfortunately is often not the case. For example, while restriction to subgroups is always a surjection for $A(G)$, it may fail to be so for $ZA(G)$. Take the symmetric group $S_3$ and the cyclic group $C_3$, where $C_3$ is viewed as a subgroup of $S_3$ as in the following figures.
\begin{figure}[ht]
\begin{center}
\begin{tabular}{|l|c|c|c|}
  \hline
    &  $()$ & $(1 2), (2 3), (1 3)$ & $(1 2 3), (1 3 2)$ \\ \hline
  $\chi_0$	& $1$ & $1$ & $1$ \\
  $\chi_1$	& $1$ & $-1$	 &  $1$  \\
  $\chi_2$	& $2$ & $0$ &  $-1$ \\
  \hline
\end{tabular}
\end{center}
\caption{Character table of $S_3$}
\label{fig:chartable_S3}
\end{figure}
\begin{figure}[ht]
\begin{center}
\begin{tabular}{|l|c|c|c|}
  \hline
    &  $()$ & $(1 2 3)$ & $(1 3 2)$\\ \hline
  $\rho_0$	& $1$ & \hfil$1$ & \hfil$1$ \\
  $\rho_1$	& $1$ & $e^{\frac{2\pi i}{3}}$	 &  $e^{\frac{4\pi i}{3}}$  \\
  $\rho_2$	& $1$ & \hfil$e^{\frac{4\pi i}{3}}$ &  $e^{\frac{2\pi i}{3}}$ \\
  \hline
\end{tabular}
\end{center}
\caption{Character table of $C_3$}
\label{fig:chartable_C3}
\end{figure}
\\
Note that $AMZA(S_3) = \text{span}\{\chi_0, \chi_1, \chi_2\}$ and $AMZA(C_3) = \text{span} \{\rho_0, \rho_1, \rho_3\}$. The restriction of $\chi_0$ and $\chi_1$ to $C_3$ both yield $\rho_0$, so it follows that the restriction map from $AMZA(S_3)$ to $AMZA(C_3)$ is not surjective. Recall that restriction being surjective was a key to the proof that $AM(A(N)) \leq AM(A(G))$ if $N$ is a subgroup of $G$, and as illustrated in the next example this failing is enough for $AMZA$ to not enjoy the same hereditary propery.
\begin{example}
For the rest of this paper we will refer to several groups using the Small Groups library implemented in GAP \cite{GAP4}, where the notation $\textrm{SmallGroup}(x,y)$ refers to the $y$th group of order $x$. Now let $G = \textrm{SmallGroup}(32,43)$, which can also be viewed as a semidirect product of $C_8$ and $C_2 \times C_2$, and let $N= D_{8}$ be the dihedral group of order $16$ identified as a normal subgroup of $G$. Calculations with GAP yield that $AMZA(G) = 2.59375$ and $AMZA(N) = 2.6875$, so $AMZA(G) < AMZA(N)$.
\end{example}
We will consider the question of whether $AMZA$ respects quotients by normal subgroups in section \ref{Prop Q}. To close this section, we provide a table reviewing known amenability bounds for some classes of related Banach algebras. In this context $A_{cb}(G)$ refers to the closure of $A(G)$ in the completely bounded multipler norm.
\renewcommand*{\arraystretch}{2}
\begin{table}[ht]
\centering
\begin{tabular}{|c|c|c|c|c|}
\hline
Banach algebra & $ZL^1(G)$ & $A(G)$ & $A_{cb}(G)$  & $ZA(G)$ \\
\hline
     Restriction on $G$ & Finite &  Locally Compact  & Locally compact & Finite \\
     \hline
    Best amenability bound & $\dfrac{7}{4}$ & $\dfrac{3}{2}$  & $\dfrac{9}{7}$ & $\dfrac{2}{\sqrt{3}}$  \\
    \hline
     Bound is sharp &    Yes&   Yes  & Unknown & Unknown\\
     \hline
     Reference & \cite[Theorem 1.2]{Choi1} & \cite[Theorem 1.6]{Choi2} & \cite[Theorem 4.6]{Juselius} & \cite[Corollary 4.3]{Alag2}\\
     \hline
\end{tabular}
\end{table}
\section{Frobenius Groups with Abelian Factor and Kernel}
It is of interest to try to determine when $AMZA(G) = AMZL(G)$. While this is not necessarily the case (for example, $AMZA(SL(2,\mathbb{F}_3)) = 4.875$ and $AMZL(SL(2,\mathbb{F}_3)) = 5$), it often does hold (See Question \ref{question: AMZA =/ AMZ} for a further discussion). In this section we will demonstrate that the amenability constants of $ZA(G)$ and $ZL^1(G)$ always agree for a particular class of semidirect products.
\begin{definition}
We say that a finite group $G$ is a Frobenius group if it has a finite, proper, non-trivial subgroup $H$ that satisfies $H \cap gHg^{-1} = \{e\}$ for all $g \in G \setminus H$. It can be shown that $K = \left(G \setminus \bigcup_{g \in G} gHg^{-1} \right) \cup \{e\}$ is a subgroup of $G$ and that $G = K \rtimes H$. We call $H$ the Frobenius complement of $G$ and $K$ the Frobenius kernel.
\end{definition}
We will consider the case when $K$ and $H$ are both abelian, which is a class of groups that includes the dihedral groups $D_{2n}$, where $n$ is odd, and the affine group of the finite field of order $q$, Aff$(\mathbb{F}_q)$, where $q$ is an odd prime power.
\begin{theorem}\label{Frobenius}
Let $G = K \rtimes H$ be Frobenius, where $K$ and $H$ are abelian and have orders $k$ and $h$ respectively. Then $AMZA(G) = AMZL(G) = 1 +  \frac{2(h^2 - 1)}{h} \left(1 - \frac{h-1}{k}\right)\left(1 - \frac{1}{k}\right)$.
\end{theorem}
\begin{proof}
By \cite[Proposition 3.3]{Alag1} we know that $\textrm{Irr}(G)$ is comprised of $h$ linear characters (the set of which we will designate $\mathfrak{L}$) that come from composition of characters from $\textrm{Irr}(H)$ and the quotient map $G \rightarrow G/K \cong H$ and $\frac{k-1}{h}$ many characters of degree $h$ induced from characters in $\textrm{Irr}(K)$. Furthermore, $G$ has trivial centre, $\frac{k-1}{h}$ conjugacy classes of size $h$ (which are all contained in $K$) and $h-1$ conjugacy classes of size $k$. Let $B_1 = \{Z(G)\}, B_2 = \{C \in \text{Cong}(G): |C| = h\},  B_3 = \{C \in \text{Cong}(G): |C| = k\}$. We know by \cite{Alag1} that $\textrm{ass}(G) = h^2 - \frac{(h^2-1)(1 + h(k-1) + (h-1)k^2}{hk^2}$, so it suffices to calculate $AMZA_{\textrm{off}}(G)$. Let $\chi \neq \chi' \in \textrm{Irr}(G)$. Then by Schur orthogonality we have that
\begin{align*}
\left|\sum_{C \in \textrm{Conj}(G)} |C|^2 \chi(C)\overline{\chi'(C)}\right| &= \left|d_\chi d_{\chi'} + h \sum_{C \in B_2} h \chi(C)\overline{\chi'(C)} + k^2 \sum_{C \in B_3} \chi(C)\overline{\chi'(C)}\right|\\
&= \left|(1-h)d_\chi d_{\chi'} + (k^2 - hk) \sum_{C \in B_3} \chi(C)\overline{\chi'(C)} + h \sum_{C \in \textrm{Conj}(G)} |C| \chi(C)\overline{\chi'(C)}\right|\\
&= \left|(1-h)d_\chi d_{\chi'} + (k^2 - hk) \sum_{C \in B_3} \chi(C)\overline{\chi'(C)}\right|.\\
\end{align*}
If $\chi \notin \mathfrak{L}$ then $\chi$ is induced from an irreducible character on $K$. By definition of induction of characters (see \cite[Chapter 5]{Isaacs}) $\chi$ vanishes on $G \setminus \bigcup_{x \in G}xKx^{-1}$, so $\chi$ vanishes on $H\setminus\{e\}$. This allows us to see that 
\begin{align*}
\sum_{\chi \text { or } \chi' \notin \mathfrak{L}, \chi \neq \chi'} d_\chi d_{\chi'}\left|\sum_{C \in \textrm{Conj}(G)} |C|^2 \chi(C)\overline{\chi'(C)}\right| &= (h-1)\sum_{\chi \text{ or } \chi' \notin \mathfrak{L}, \chi \neq \chi'} d_\chi^2 d_{\chi'}^2\\
 &= (h-1)\sum_{\chi, \chi' \notin \mathfrak{L}, \chi \neq \chi'} d_\chi^2 d_{\chi'}^2 + 2(h-1)\sum_{\chi \notin \mathfrak{L}, \chi' \in \mathfrak{L}} d_\chi^2 d_{\chi'}^2\\
  &= (h-1)\sum_{\chi, \chi' \notin \mathfrak{L}, \chi \neq \chi'} d_\chi^2 d_{\chi'}^2 + 2(h-1)\sum_{\chi \notin \mathfrak{L}, \chi' \in \mathfrak{L}} d_\chi^2 d_{\chi'}^2\\
  &= (h-1) \left(\frac{k-1}{h}\right)\left(\frac{k-1}{h} - 1\right)h^4 + 2(h-1)h\left(\frac{k-1}{h}\right)h^2.
\end{align*}
On the other hand, if $\chi, \chi' \in \mathfrak{L}$ with $\chi \neq \chi'$ then $\chi_H, \chi'_H \in \textrm{Irr}(H)$ so it follows by Schur Orthogonality that
\[
\sum_{C \in B_1 \cup B_3} |C| \chi(C)\overline{\chi'(C)}= \sum_{C \in \textrm{Conj}(H)} |C| \chi_H(C)\overline{\chi'_H(C)} = 0.
\]
Thus it follows that 
\begin{align*}
    \sum_{\chi, \chi' \mathfrak{L}, \chi \neq \chi'} d_\chi d_{\chi'}\left|\sum_{C \in \textrm{Conj}(G)} |C|^2 \chi(C)\overline{\chi'(C)}\right| &= \sum_{\chi, \chi' \mathfrak{L}, \chi \neq \chi'} d_\chi d_{\chi'} \left|(1 -h - k^2 + hk) + (k^2 - hk) \sum_{C \in B_1 \cup B_3} \chi(C)\overline{\chi'(C)}\right|\\
    &= h(h-1)(h+k^2 - hk -1).
\end{align*}
Combining everything together, we have that
\begin{align*}
    &AMZA_{\textrm{off}}(G) \\
    &= \frac{1}{h^2k^2}\left[(h-1) \left(\frac{k-1}{h}\right)\left(\frac{k-1}{h} - 1)h^4 + 2(h-1\right)h\left(\frac{k-1}{h}\right)h^2 + h(h-1)(h+k^2 - hk -1)\right].
\end{align*}
We can now see that 
\[
    AMZA_{\textrm{off}}(G) + \textrm{ass}(G) = \frac{2h^2}{k^2} - \frac{2h^2}{k} - \frac{2h}{k^2} + \frac{2}{hk^2} + 2h - \frac{2}{h} - \frac{2}{k^2} + \frac{2}{k} + 1 = 1 + 2 \frac{h^2 - 1}{h} \left(1 - \frac{h-1}{k}\right)\left(1 - \frac{1}{k}\right).
\]
By appealing again to \cite{Alag1}, we get that $AMZA(G) = AMZL(G)$.
\end{proof}
\section{Structure of Sum}
We wish to learn more about the behavior of 
\[
AMZA(G) = \frac{1}{|G|^2}\sum_{\chi, \chi' \in \textrm{Irr}(G)}d_\chi d_{\chi'} \left |\sum_{C \in \textrm{Conj}(G)} |C|^2 \chi(C)\overline{\chi'(C)}\right|.
\]
Of particular interest is the inside sum. Because irreducible characters have values in the algebraic integers, we know that $\left |\sum\limits_{C \in \textrm{Conj}(G)} |C|^2 \chi(C)\overline{\chi'(C)}\right| \in \mathbb{Z}$, although it turns out that this is still true even without taking the complex magnitude. 
First we need a lemma that we will use several times.
\begin{lemma}\label{Z(G)}
Let $\phi_1, \phi_2,..., \phi_{|Z(G)|}$ be the distinct irreducible characters of $Z(G)$ for a finite group $G$. Let $\chi \in \textrm{Irr}(G)$ and let $\chi_{Z(G)}$ denote the restriction of $\chi$ to $Z(G)$. Then there exists $\phi_i$ such that $\chi_{Z(G)} = d_\chi \phi_i$. This allows us to define the pairwise-disjoint sets $A_i = \{\chi \in \textrm{Irr}(G): \chi_{Z(G)} = d_\chi \phi_i\}$. Furthermore, it is known that $\sum_{\chi \in A_i} d_\chi^2 = |G: Z(G)|$.
\end{lemma}
\begin{proof}
This follows by Clifford's Theorem \cite[Theorem 6.2]{Isaacs}.
\end{proof}
\begin{proposition}\label{proposition: Integer}
The value $\sum\limits_{C \in \textrm{Conj}(G)} |C|^2 \chi(C)\overline{\chi'(C)}$ is an integer divisible by $|Z(G)|$.
\end{proposition}
\begin{proof}
For convenience let $Z = Z(G)$. Let $\textrm{Irr}(Z) = \{\phi_1,..,\phi_{|Z|}\}$ and let $A_i$ be the decomposition from Lemma \ref{Z(G)}. Let $\chi, \chi' \in \textrm{Irr}(G)$. By the lemma there exists $A_i$ and $A_j$ such that $\chi \in A_i$ and $\chi' \in A_j$. Let $|g^G|$ denote the size of the conjugacy class of $g$ in $G$. We adapt the argument from \cite{Sambale} to see the following:
\begin{align*}
    \sum_{C \in \textrm{Conj}(G)}|C|^2 \chi(C) \overline{\chi'(C)} &= \sum_{gZ \in G/Z} \frac{1}{d_\chi d_{\chi'}} |g^G|\chi(g)\overline{\chi'(g)} \sum_{z \in Z} \chi(z) \overline{\chi'(z)}\\
    &= \langle \chi_Z, \chi'_Z \rangle\cdot |Z| \sum_{gZ \in G/Z} \frac{1}{d_\chi d_{\chi'}} |g^G| \chi(g) \overline{\chi'(g)}\\
    &= \delta_{ij} \cdot |Z| \sum_{gZ \in G/Z}  |g^G| \chi(g) \overline{\chi'(g)}.\\
\end{align*}
Similarly as noted in the proof of \cite[Proposition 1]{Sambale}, $\sum\limits_{gZ \in G/Z}  |g^G| \chi(g) \overline{\chi'(g)}$ must be a rational algebraic integer, hence an integer. 
\end{proof}
This yields the interesting fact that $AMZA(G)\cdot|G|^2$ is divisible by $|Z(G)|$. We also get the following rough estimate of $AMZA(G/Z(G))$ compared to $AMZA(G)$.
\begin{corollary}
$AMZA(G) \geq \dfrac{AMZA(G/Z(G))}{|Z(G)|}$.
\end{corollary}
\begin{proof}
For notational convenience let $Z = Z(G)$. By the proof of Proposition \ref{proposition: Integer} we have that
\[
AMZA(G) = \frac{|Z|}{|G|^2} \sum_{i=1}^{|Z|}\sum_{\chi, \chi' \in A_i}d_\chi d_{\chi'} \left | \sum_{gZ \in G/Z} |g^G| \chi(g) \overline{\chi'(g)}\right|.
\]
We can identify $A_1$ with $\textrm{Irr}(G/Z)$, and then using the fact that $|g^G| \geq |{gZ}^{G/Z}|$ it follows that $AMZA(G) \geq |Z| \cdot AMZA(G/Z)$, as desired.
\end{proof}
\section{Groups with two conjugacy class sizes}
\begin{definition}
Let $cd(G)$ denote the set of dimensions of irreducible characters of $G$, and let $cc(G)$ denote the set of lengths of conjugacy classes of $G$. We say that $G$ has two characters degrees (or two conjugacy clas sizes) if $|cd(G)| = 2$ (or $|cc(G)| = 2$).
\end{definition}
There is a nice formula for $AMZL(G)$ in the two character degree case:
\begin{theorem}\label{Theorem: two character degrees}\cite[Theorem 2.4]{Alag1}. 
Let $G$ be a finite group with $cd(G) = \{1,m\}$. Then
\[
AMZL(G) = 1 + 2(m^2 - 1)\left(1 - \frac{1}{|G| \cdot |G'|} \sum_{C \in \textrm{Conj}(G)} |C|^2 \right).
\]
\end{theorem}
We can prove a dual formula for $AMZA(G)$ if instead we assume that there are only two possible sizes of conjugacy classes. First, we show a lemma.
\begin{lemma}
Let $G$ be a finite group. Then 
\[
\frac{1}{|G|^2}\sum_{\chi, \chi' \in \textrm{Irr}(G)}d_\chi d_{\chi'}\left|\sum_{x \in Z(G)}\chi(x)\overline{\chi'(x)}\right| = 1.
\]
\end{lemma}
\begin{proof}
Using the notation from Lemma \ref{Z(G)}, let $\chi \in A_i$ and $\chi' \in A_j$. Then by Schur orthogonality for $Z(G)$
\[
\left|\sum_{x \in Z(G)} \chi(x)\overline{\chi(x)}\right| = |Z(G)| \cdot \langle d_\chi \phi_i | d_{\chi'} \phi_j\rangle_{Z(G)} = |Z(G)| \cdot d_\chi d_{\chi'} \delta_{ij},
\]
so it follows that
\begin{align*}
    \frac{1}{|G|^2}\sum_{\chi, \chi' \in \textrm{Irr}(G)}d_\chi d_{\chi'}\left|\sum_{x \in Z(G)}\chi(x)\overline{\chi'(x)}\right| &= \frac{1}{|G|^2}\sum_{i=1}^{|Z(G)|} \sum_{\chi,\chi' \in A_i} |Z(G)| d_\chi^2 d_{\chi'}^2 \\
    &= \frac{1}{|G|^2}\sum_{i=1}^{|Z(G)|} |Z(G)| \cdot |G: Z(G)|^2 \\
    &= 1.
\end{align*}
\end{proof}
\begin{theorem}\label{Theorem: two conjugacy class sizes}
Let $G$ be a finite group with $cc(g) = \{1,s\}$. Then
\[
AMZA(G) = 1 + 2(s-1)\left(1 - \frac{1}{|G| \cdot |\mathfrak{L}(G)|}\sum_{\chi \in \textrm{Irr}(G)}d_\chi^4\right),
\]
where $\mathfrak{L}(G)$ is the set of linear irreducible characters of $G$.
\end{theorem}
\begin{proof}
If $s = 1$ then $G$ is abelian, so $AMZA(G) = AMZL(G) = 1$, hence we can assume that $s > 1$. Recall that we have the decomposition $AMZA(G) = AMZA_{\textrm{off}}(G) + \textrm{ass}(G)$. We will work with each component separately.
\begin{align*}
|G|^2 AMZA_{\textrm{off}}(G) &= \sum_{\chi \neq \chi' \in \textrm{Irr}(G)}d_\chi d_{\chi'} \left|\sum_{C \in \textrm{Conj}(G)} |C|^2 \chi(C)\overline{\chi'(C)}\right|\\
&= \sum_{\chi \neq \chi' \in \textrm{Irr}(G)}d_\chi d_{\chi'} \left|\sum_{x \in Z(G)} \chi(x) \overline{\chi'(x)} + s\sum_{C \in \textrm{Conj}(G), |C| > 1} s \chi(C)\overline{\chi'(C)}\right|\\
&= \sum_{\chi \neq \chi' \in \textrm{Irr}(G)}d_\chi d_{\chi'} \left|(1-s)\sum_{x \in Z(G)} \chi(x) \overline{\chi'(x)} + s\sum_{C \in \textrm{Conj}(G)} s \chi(C)\overline{\chi'(C)}\right|\\
&= \sum_{\chi \neq \chi' \in \textrm{Irr}(G)}d_\chi d_{\chi'} \left|(1-s)\sum_{x \in Z(G)} \chi(x) \overline{\chi'(x)}\right|\\
&= (s-1)\sum_{\chi, \chi' \in \textrm{Irr}(G)}d_\chi d_{\chi'} \left|\sum_{x \in Z(G)} \chi(x) \overline{\chi'(x)}\right| - (s-1)\sum_{\chi \in \textrm{Irr}(G)}d_\chi^2 \sum_{x \in Z(G)} |\chi(x)|^2\\
&= |G|^2(s-1) + (1-s) \cdot |Z(G)| \cdot \left(\sum_{\chi \in \textrm{Irr}(G)}d_\chi^4\right).\\
\end{align*}
And then for $\textrm{ass}(G)$ we have
\begin{align*}
|G|^2 \textrm{ass}(G) &= \sum_{\chi \in \textrm{Irr}(G)}d_\chi^2 \sum_{C \in \textrm{Conj}(G)} |C|^2 |\chi(C)|^2\\
&= \sum_{\chi \in \textrm{Irr}(G)}d_\chi^2\left(s|G| + (1-s)\sum_{x \in Z(G)} |\chi(x)|^2 \right)\\
&= s|G|^2  + (1-s) \cdot |Z(G)| \cdot \left(\sum_{\chi \in \textrm{Irr}(G)}d_\chi^4\right).\\
\end{align*}
Adding them together and rearranging terms, we achieve formula
\[
AMZA(G) = 1 + 2(s-1)\left(1 - \frac{|Z(G)|}{|G|^2}\sum_{\chi \in \textrm{Irr}(G)}d_\chi^4\right).
\]
Finally, we recall the fact that $|\mathfrak{L}(G)| = \frac{|G|}{|Z(G)|}$, to get the desired result.
\end{proof}
\begin{example}
\label{remark: p extraspecial}
Let $p$ be a prime. A group $G$ is called p-extraspecial if $|G| = p^{2n+1}$ for some integer $n$, $|Z(G)| = p$, and $G/Z(G)$ is a non-trivial elementary abelian p-group. As noted in \cite{Alag1}, such groups have two conjugacy class sizes and two character degrees, so both of the above formulas apply and yield the same result, namely that 
\[
AMZL(G) = AMZA(G) = 1 +2 \left(1 - \frac{1}{p^{2n}}\right) \left(1 - \frac{1}{p}\right).
\]
\end{example}
This leads to the question: will $AMZL(G) = AMZA(G)$ always hold in the case that $G$ has two conjugacy class sizes and two character degrees? Based on our earlier formulas, Y. Choi has observed that the answer is
positive.
\begin{theorem}\label{Theorem: cd and cc = 2}
Let $G$ be a finite group where $cd(G) = \{1, m\}$ and $cc(G) = \{1,s\}$. Then 
\[
AMZA(G) = AMZL(G)
\]
\end{theorem}
The following proof is based on a personal communication from Choi.
\begin{proof}[Theorem 4.6]
If either $s = 1$ or $m = 1$ then they both equal $1$, in which case $G$ is abelian and the result is trivial. Instead, we assume that $m,s > 1$. For notational convenience, let $k = |\textrm{Irr}(G)| = |\textrm{Conj}(G)|$, $Z = Z(G)$, and $\mathfrak{L} = \mathfrak{L}(G)$.
Because $|G| = \sum_{C \in \textrm{Conj}(G)} |C| = |Z| + (k-|Z|)s$ we can see that
\begin{align*}
    \sum_{C \in \textrm{Conj}(G)} |C|^2 &= |Z| + (k - |Z|)s^2\\
    &= |Z| + s(|Z| + (k-|Z|)s) - s|Z|\\
    &= |Z| + s|G| - s|Z|\\
    &= |Z| + (k - |Z|)s + s|G| - sk\\
    &= (s+1)|G| - sk
\end{align*}
Similarly, from $|G| = \sum_{\chi \in \textrm{Irr}(G)} {d_\chi}^2 = |\mathfrak{L}| + (k - |\mathfrak{L}|)m^2$ it follows that
\begin{align*}
    \sum_{\chi \in \textrm{Irr}(G)} {d_\chi}^4 &= |\mathfrak{L}| + (k - |\mathfrak{L}|)m^4\\
    &= |\mathfrak{L}| + m^2(|\mathfrak{L}| + (k - |\mathfrak{L}|)m^2) - m^2|\mathfrak{L}|\\
    &= |\mathfrak{L}| + m^2 |G| - m^2|\mathfrak{L}|\\
    &= |\mathfrak{L}| + (k-|\mathfrak{L}|)m^2 + m^2|G| - km^2\\
    &= (m^2+1)|G| - km^2
\end{align*}
Define the function 
\[
f(x,y) = x - 1 - \frac{xk - |G|}{|G|}\Big((y + 1)|G| - yk \Big).
\]
The above calculations combined with Theorem \ref{Theorem: two character degrees} and Theorem \ref{Theorem: two conjugacy class sizes} yields that $AMZL(G) -1 = 2 f(m^2, s)$ and $AMZA(G) - 1 = 2 f(s,m^2)$. However, we can see that
\begin{align*}
    f(x,y) &= x - 1 + y + 1 - \frac{xk(y+1)}{|G|} + \frac{(xk - |G|)(yk)}{|G|^2}\\
    &= x + y - \frac{xy + x + y}{|G|} + \frac{xyk^2}{|G|^2}.
\end{align*}
In particular note that $f(x,y) = f(y,x)$, thus $AMZA(G) = AMZL(G)$, as desired.
\end{proof}
\begin{remark}
By results in \cite{Fernandez}, for any integer $n$ there exists groups $G$ and $H$ such that $|cd(G)| = |cc(H)| = 2$ and $|cc(G)| = |cd(H)| = n$. This tells us that the two conjugacy class size and two character degree conditions are possibly independent of each other, so there is no reason to expect that the conclusion of Theorem \ref{Theorem: cd and cc = 2} will hold if only one of the sets $cd(G)$ and $cc(G)$ has size $2$. Indeed, $G = \textrm{SmallGroup}(256,10070)$ is an example of a group with $|cc(G)| = 2$, $|cd(G)| = 3$, and $AMZL(G) \neq AMZA(G)$.
\end{remark}

\section{AMZA of Quotient Groups}\label{Prop Q}
Recall that an essential ingredient in Choi's \cite{Choi1} proof that $\frac{7}{4}$ is the sharp amenability bound for $ZL^1(G)$ is the fact that $AMZL(G) \geq AMZL(G/N)$ for $N \trianglelefteq G$. If we look at the collection of groups such that $AMZA(G)$ respects all possible quotients, then by utilizing similar techniques as in \cite{Choi1} we can prove that $\frac{7}{4}$ is a sharp amenability bound.
\begin{theorem}\label{Theorem: Q}
$\frac{7}{4}$ is the sharp amenability bound over the collection of finite groups $G$ with the property that $AMZA(G) \geq AMZA(G/N)$ for all $N \trianglelefteq G$.
\end{theorem}
\begin{proof}
By taking sufficiently large enough quotients of $G$, we can assume without loss of generality that $G$ is non-abelian but possesses no non-abelian proper quotients. As noted in \cite[Lemma 4.4 and Theorem 4.5]{Choi1}, there are three possibilities:
\begin{itemize}
    \item $G$ has a non-trivial centre
    \item $G$ has a trivial centre and a conjugacy class of size $2$
    \item $G$ has a trivial centre and no conjugacy classes of size $2$.
\end{itemize}
The first two options correspond with $G$ either being a two conjugacy class size and two character degree group, or being isomorphic $D_{2p}$ for some odd prime $p$. Theorem \ref{Theorem: cd and cc = 2} and Theorem \ref{Frobenius} apply respectively, which shows that $AMZL(G) = AMZA(G)$ in those cases. If $G$ has a trivial centre and no conjugacy classes of size $2$ then \cite[Proposition 4.12]{Choi1} yields that $AMZA(G) \geq \textrm{ass}(G) \geq \frac{7}{4}$.
\end{proof}
 It can be shown computationally that all groups of order less than $192$ satisfy the conditions of Theorem \ref{Theorem: Q}. The above prompts the question: can this argument apply to every finite group? As shown by the next example, this is not the case.
\begin{example}\label{example: quotient}
Let $G = \text{SmallGroup}(192,1022)$ and choose $N \cong C_2$ in $G$ such that $G/N \cong \text{SmallGroup}(96,204)$, then $AMZA(G) = 13.4921875$ and $AMZA(G/N) = 15.53125$. This example also demonstrates that auxiliary minorant of $AMZ$ doesn't always respect quotients. For these choices of $G$ and $N$ we have that $\textrm{ass}(G) = 7.2109375$ and $\textrm{ass}(G/N) = 8.265625$.
\end{example}

\section{Further Questions to Explore}
\begin{question}\label{question: AMZA =/ AMZ}
For which groups does $AMZA(G) \neq AMZL(G)$ hold? 
\end{question}
This behavior seems to be relatively uncommon and only appears to be possible for specific orders. Calculations in GAP show that of the 851 non-abelian groups with order less than $100$ only $173$ of them have the property that $AMZL(G) \neq AMZL(G)$, and the only orders achieved by these groups are $24, 48, 60, 64, 72, 80, 96$. Interestingly the smallest odd order group with differing amenability constants for $ZA(G)$ and $ZL^1(G)$ is SmallGroup(567, 16). It appears that $AMZA(G) \neq AMZL(G)$ often holds when $G$ is a non-monomial group (that is, $G$ contains an irreducible character that is not induced from a linear character on a subgroup). There are $24$ non-monomial groups of order less than $100$, and $21$ of them have the property that $AMZA(G) \neq AMZL(G)$. 
\begin{question}
Does $AMZA(G/N) \leq AMZA(G)$ hold if $N$ is a normal Hall subgroup of $G$?
\end{question}
\begin{question}
For which groups is $AMZA(G) = \frac{7}{4}$ actually achieved? Is it necessary that $G$ be nilpotent?
\end{question}
\begin{question}
Is it true that $AMZA(G') \leq AMZA(G)$, where $G'$ is the derived subgroup of $G$?
\end{question}
\begin{question}
Are there other gaps respected by amenability constants besides around $1$? 
\end{question}
\subsection*{Acknowledgements}
The author is thankful to Yemon Choi for helpful discussions about these topics during the author's trip to the University of Lancaster in July 2022.

\bibliography{references}
\bibliographystyle{amsplain}
\renewcommand*{\bibname}{References}

\end{document}